\documentclass{amsproc}
\usepackage{epsfig,amscd,amssymb,amsmath,amsfonts}
\usepackage{amssymb}
\usepackage{epsfig}
\newtheorem{theorem}{Theorem}[section]

\newtheorem{proposition}{Proposition}[section]
\theoremstyle{definition}

\newtheorem{example}[theorem]{Example}

\theoremstyle{remark}
\newtheorem{remark}[theorem]{Remark}

\numberwithin{equation}{section}

\begin{document}

\title{Secondary  Hochschild Cohomology}
\author{Mihai D. Staic}
\address{
Department of Mathematics and Statistics, Bowling Green State University, Bowling Green, OH 43403 }
\address{Institute of Mathematics of the Romanian Academy, PO.BOX 1-764, RO-70700 Bu\-cha\-rest, Romania.}
\thanks{This work was partially supported by a grant of the Romanian National Authority for Scientific Research, CNCS-UEFISCDI, project number PN-II-ID-PCE-2011-3-0635, contract nr. 253/5.10.2011.}
%\curraddr{}
\email{mstaic@gmail.com}


%\author{Author Two}
%\address{Mathematical Research Section, School of Mathematical Sciences,
%Australian National University, Canberra ACT 2601, Australia}

%    General info
\date{January 1, 1994 and, in revised form, June 22, 1994.}

%\dedicatory{I thank   }


\begin{abstract}  In this paper we define a new cohomology theory for a $B$-algebra $A$. We use this cohomology to study  deformations of algebras $A[[t]]$, that have a  $B$-algebra structure.
\end{abstract}

\maketitle

%\section*{This is an unnumbered first-level section head}

%
%%%%%%%%%%%%%%%%%%%%%%%%%%%%%%%%%%%%%%%%%%%%%%%%%%%

%%%%%%%%%%%%%%%%%%%%%%%%%%%%%%%%%%%%%%%%%%%%%%%%%%%%
\section*{Introduction}

Hochschild cohomology is used  to describe deformations of algebras. More precisely, let $A$ be an algebra, and  $m_t:A[[t]]\otimes A[[t]]\to A[[t]]$ a product that is determined by $m_t(a\otimes b)=ab+c_1(a\otimes b)t+c_2(a\otimes b)t^2+...$, where $c_i:A\otimes A\to A$. It is known from  \cite{g} that if the  map $m_t$ is associative $mod\; t^2$, then  $c_1$ is a two cocycle and the obstruction for $m_t$ to be associative $mod\; t^n$ is an element in $H^3(A,A)$. If $A$ is a $B$-algebra and the above maps $c_i$ are $B$-linear, one can give a  trivial $B$-algebra structure on $A[[t]]$. 
In this paper we are interested in non-trivial $B$-algebra structures on $A[[t]]$. 

First, we notice that having $\varepsilon:B\to A$ that gives a $B$-algebra structure on $A$, is equivalent with having a family of products $\{m_{\alpha}\}_{\alpha\in B}$, $m_{\alpha}:A\otimes A\to A$, ($m_{\alpha}(a\otimes b)=\varepsilon(\alpha)ab$) that satisfies the following generalized associativity condition $m_{\alpha\beta}(id\otimes m_{\gamma})=m_{\beta\gamma}(m_{\alpha}\otimes id)$. We consider a deformation of this family of products to  $m_{\alpha,t}:A[[t]]\otimes A[[t]]\to A[[t]]$, $m_{\alpha,t}(a\otimes b)=\varepsilon(\alpha)ab+c_1(a\otimes b\otimes \alpha)t+c_2(a\otimes b\otimes \alpha)t^2+...$, where  $c_i:A\otimes A\otimes B\to A$. 
 We  introduce a new cohomoogy theory called the secondary cohomology and use it to study under what condition the products  $\{m_{\alpha ,t}\}_{\alpha\in B}$  give a $B$-algebra structure on $A[[t]]$.  More precisely, if $A$ and $B$ are $k$-algebras, $B$ is commutative, $\varepsilon:B\to A$ is a morphism of $k$-algebras such that $\varepsilon(B)\subset {\mathcal Z}(A)$, and $M$ is an $A$-bimodule,  we construct  cohomology groups $H^n((A,B,\varepsilon);M)$. When $B=k$ we have that $H^n((A,k,\varepsilon);M)$ is the usual Hochschild cohomology. 

We prove that  the family of multiplications $\{m_{\alpha,t}\}_{\alpha\in B}$  satisfies the condition $m_{\alpha\beta,t}(id\otimes m_{\gamma,t})=m_{\beta\gamma,t}(m_{\alpha,t}\otimes id)$ $mod\;t^2$ if and only if $c_1$ is a 2-cocycle.  Moreover the obstruction for $m_{\alpha,t}$ to be associative $mod\; t^n$ is an element in $H^3((A,B,\varepsilon),A)$. In particular, we give conditions under which we have a $B$-algebra structure on $A[[t]]$. In general this is not the trivial multiplication with elements in $B$. 

In the last section we define a secondary cyclic module for a commutative $k$-algebra $B$, and discuss the relation with the usual cyclic module associated to a $k$-algebra.

\section{Preliminaries}

We recall from \cite{g}, \cite{gs} and \cite{lo} a few facts about deformations of algebras and Hochschild cohomology. 

In this paper $k$ is a field and $\otimes=\otimes_k$. Suppose that $A$ is an associative $k$-algebra (not neccesarly commutative), and $M$ is an $A$-bimodule. Define  $C^n(A, M)=Hom_k(A^{\otimes n}, M)$, and $\delta_n:C^n(A,M)\to C^{n+1}(A,M)$ determined by:
\begin{eqnarray*}
&&\delta_n(f)(a_0\otimes a_1\otimes ...\otimes a_n)=a_0f(a_1\otimes ...\otimes a_n)+\sum_{i=1}^n(-1)^if(a_0\otimes ...\otimes a_{i-1}a_{i}\otimes \\&& ...\otimes a_n)+
(-1)^{n+1}f(a_0\otimes ...\otimes a_{n-1})a_n.
\end{eqnarray*}
One can show that $\delta_{n+1}\delta_n=0$.
The homology of this complex is denoted by $H^n(A, M)$ and is called the Hochschild cohomology of $A$ with coefficients in $M$. 

Suppose that $m_t:A[[t]]\otimes A[[t]]\to A[[t]]$, where $$m_t(a\otimes b)=ab+c_1(a\otimes b)t+c_2(a\otimes b)t^2+...$$
It is known that if $m_t$ is associative  $mod\; t^2$, then $c_1$ is a two cocylce and the isomorphism type of $(A[[t]],m_t)$ determines a unique element in $H^2(A,A)$.  If $m_t$ is associative $mod \;t^n$, then the obstruction for $m_t$ to be associative mod $t^{n+1}$ is a element in $H^3(A,A)$. More precisely, we must have that $\sum_{p+q=n}c_p(c_q(a\otimes b)\otimes c)-c_p(a\otimes c_q(b\otimes c))$ must be zero as an element in $H^3(A,A)$ (see \cite{gs} for details).

We recall from \cite{may} and \cite{w} the construction of the simplicial group $K(G,2)$. Let $G$ be an abelian group and define
$$K_q=G^{\frac{q(q-1)}{2}}.$$
The elements of  $G^{\frac{q(q-1)}{2}}$ are $\frac{q(q-1)}{2}$-tuples $(g_{u,v})_{(0\leq u<v\leq q-1)}$ with the index in the lexicographic order:
$$(g_{0,1},g_{0,2},...,g_{0,q-1},g_{1,2},g_{1,3},...,g_{1,q-1},...,g_{q-2,q-1}).$$
For every $0\leq i\leq q$ we define
$$\partial_i:K_q=G^{\frac{q(q-1)}{2}}\to K_{q-1}=G^{\frac{(q-1)(q-2)}{2}},$$
$\partial_i((g_{u,v})_{(0\leq u<v\leq q-1)})=(h_{u,v})_{(0\leq u<v\leq q-2)}$ where
$$h_{u,v}= \left\{
\begin{array}{c l}
g_{u,v}\; &{\rm if}\; 0\leq u<v < i-1\\
g_{u,v}g_{u,v+1}\; &{\rm if}\; 0\leq u<v= i-1\\
g_{u,v+1}\; &{\rm if}\; 0\leq u< i-1 < v\\
g_{u,v+1}g_{u+1,v+1}\; &{\rm if}\; 0\leq u= i-1 < v\\
g_{u+1,v+1}\; &{\rm if}\; i-1<u <v.
\end{array}
\right.
$$
Also, we define:
$$s_i: K_q=G^{\frac{q(q-1)}{2}}\to K_{q+1}=G^{\frac{(q+1)q}{2}}$$
$s_i((g_{u,v})_{(0\leq u<v\leq q-1)})=(k_{u,v})_{(0\leq u<v\leq q)}$ where

$$k_{u,v}= \left\{
\begin{array}{c l}
g_{u,v}\; &{\rm if}\; 0\leq u<v < i\\
1\; &{\rm if}\; 0\leq u<v = i\\
g_{u,v-1}\; &{\rm if}\; 0\leq u<i<v\\
g_{u,v-1}\; &{\rm if}\; 0\leq u=i< v-1\\
1\; &{\rm if}\; 0\leq u=i=v-1\\
g_{u-1,v-1}\; &{\rm if}\; 0\leq i<u <v.
\end{array}
\right.
$$

One can show that the above construction gives a  $K(G,2)$ simplicial group. If we define $\tau_q:K_q\to K_q$,
$$\tau_q((g_{u,v})_{(0\leq u<v\leq q-1)})=(l_{u,v})_{(0\leq u<v\leq q-1)}$$ where
$$l_{u,v}= \left\{
\begin{array}{c l}
g_{v-1,v}g_{v-1,v+1}...g_{v-1,q-1}g_{v,v+1}^{-1}g_{v,v+2}^{-1}...g_{v,q-1}^{-1}\; {\rm if}\; 0=u<v\\
g_{u-1,v-1}\; {\rm if}\; 0< u<v,\\
\end{array}
\right.$$
then $(K_q,\partial_i, s_i,\tau_q)$ is a cyclic simplicial group. 

In cite \cite{sm} we  replaced the group $G$  with a commutative  Hopf algebras $H$, and we defined  a secondary cyclic module for $H$. The naive generalization  to algebras does not work because the map $\tau_n$ involves the existence of an inverse (antipode) in the group $G$ (Hopf algebra $H$). In the last section we will discuss how this problem can be avoided.

\section{Associative Families of Products}

\subsection{Motivation}

In \cite{staic} we introduced the secondary cohomology of a pair $(G,H)$  with coefficients in a $G$-module $K$ ($G$ is a group and $H$ is an  abelian group). In that cohmology, a 2-cocycle is a map $f: G\times G\times H\to K$ such that $$g_1f(g_2,g_3,a_{12})-f(g_1g_2,g_3,a_{02})+f(g_1,g_2g_3,a_{01}a_{02}a_{12}^{-1})-f(g_1,g_2,a_{01})=0.$$

It is well known that if $H$ is trivial and  $K=k^*$ (with the trivial $G$-module action), then a 2-cocycle will define an associative multiplication $m:k^G\times k^G\to k^G$ determined by $m(e_g,e_h)=f(g,h)e_{gh}$ (see  for example \cite{tu}). 

If the group $H$ is non trivial, then we get a family of products $m_{a}:k^G\times k^G\to k^G$ for all $a\in H$, determined by $m_{a}(e_g,e_h)=f(g,h,a)e_{gh}$. This family satisfies the following generalized associativity condition $$m_{a_{01}a_{02}a_{12}^{-1}}(id\otimes m_{a_{12}})=m_{a_{02}}(m_{a_{01}}\otimes id).$$
If we make a change of variables $\alpha=a_{01}$, $\beta=a_{02}a_{12}^{-1}$ and $\gamma=a_{12}$ we have:
\begin{eqnarray}
m_{\alpha\beta}(id\otimes m_{\gamma})=m_{\beta\gamma}(m_{\alpha}\otimes id).
\label{genA}
\end{eqnarray}

In this paper we will study families of products that satisfy an associativity condition similar with (\ref{genA}).

\subsection{Associative Families of Products and $B$-algebras} \label{subsec2}
Suppose that $A$ is a $k$-algebra, $B$ is a commutative $k$-algebra, and $\varepsilon:B\to A$ is a morphism of $k$-algebras such that $\varepsilon(B)\subset {\mathcal Z}(A)$ (i.e. $A$ is a $B$-algebra). For each $\alpha \in B$ we have a map $m_{\alpha}:A\otimes A\to A$, defined by $$m_{\alpha}(a\otimes b)=\varepsilon(\alpha)ab=a\varepsilon(\alpha)b=ab\varepsilon(\alpha).$$ 
Once can easily check that for all $\alpha$, $\beta$, $\gamma \in B$, and $q\in k$ we have
\begin{eqnarray}\label{genA31}
m_{\alpha+\beta}(a\otimes b)=m_{\alpha}(a\otimes b)+m_{\beta}(a\otimes b),
\end{eqnarray}
\begin{eqnarray}\label{genA32}
m_{q\alpha}(a\otimes b)=qm_{\alpha}(a\otimes b),
\end{eqnarray}
\begin{eqnarray}\label{genA33}
m_{\beta\gamma}(m_{\alpha}\otimes id)=m_{\alpha\beta}(id\otimes m_{\gamma}).
\end{eqnarray}

Conversely, suppose that $B$ is a commutative $k$-algebra, $A$ is a $k$-vector space and we have a family of products $
{\mathcal M}=\{m_{\alpha}\}_{\alpha\in B}$, $m_{\alpha}:A\otimes A\to A$ such that (\ref{genA31}), (\ref{genA32}) and (\ref{genA33}) hold, and  $(A, m_{1})$ is a $k$-algebra with unit. We want to show that  $\varepsilon:B\to A$, $\varepsilon(\alpha)=m_{\alpha}(1\otimes 1)$ is a morphism of $k$-algebras and $\varepsilon(B)\subset {\mathcal Z}(A)$. 
Indeed, for all $a\in A$, $\alpha$, $\beta\in B$ and $p$, $q\in k$ we have:
\begin{eqnarray*}
\varepsilon(q\alpha+p\beta)&=&m_{q\alpha+p\beta}(1\otimes 1)\\
&=&pm_{\alpha}(1\otimes 1)+qm_{\beta}(1\otimes 1)\\
&=&p\varepsilon(\alpha)+q\varepsilon(\beta),
\end{eqnarray*}
and
\begin{eqnarray*}
\varepsilon(\alpha)\varepsilon(\beta)&=&m_{1}(m_{\alpha}(1\otimes1)\otimes m_{\beta}(1\otimes 1))\\
&=&m_{\beta}(m_{1}(m_{\alpha}(1\otimes1)\otimes 1)\otimes 1)\\
&=&m_{\beta}(m_{\alpha}(1\otimes m_1(1\otimes1))\otimes 1)\\
&=&m_{\beta}(m_{\alpha}(1\otimes 1)\otimes 1)\\
&=&m_{\alpha\beta}(1\otimes m_1(1\otimes 1))\\
&=&m_{\alpha\beta}(1\otimes 1)\\
&=&\varepsilon(\alpha\beta).
\end{eqnarray*}
So $\varepsilon$ is a morphism of $k$-algebras. We also have:
\begin{eqnarray*}
a\varepsilon(\alpha)&=&m_{1}(a\otimes \varepsilon(\alpha))\\
&=&m_{1}(a\otimes m_{\alpha}(1\otimes1))\\
&=&m_{\alpha}(m_1(a\otimes 1)\otimes1))\\
&=&m_{\alpha}(m_1(1\otimes a)\otimes1))\\
&=&m_{\alpha}(1\otimes m_1(a\otimes 1)))\\
&=&m_{\alpha}(1\otimes m_1(1\otimes a)))\\
&=&m_{1}(m_{\alpha}(1\otimes 1) \otimes a)))\\
&=&m_1(\varepsilon(\alpha) \otimes a)\\
&=&a\varepsilon(\alpha)
\end{eqnarray*}
which means that $\varepsilon(B)\subset {\mathcal Z}(A)$. To summarize we have the following result:
\begin{proposition} Let $B$ be a commutative $k$-algebra and 
${\mathcal M}=\{m_{\alpha}\}_{\alpha\in B}$, $m_{\alpha}:A\otimes A\to A$ a family of products  such that $(A, m_{1})$ is a $k$-algebra with identity. Then (\ref{genA31}), (\ref{genA32}) and (\ref{genA33}) hold if and only if $\varepsilon:B\to A$,  $\varepsilon(\alpha)=m_{\alpha}(1\otimes 1)$ gives a $B$-algebra structure on $A$.  \label{propfam}
\end{proposition}

\section{Secondary Cohomology for Algebras}

\subsection{Deformation of Families of Products}
 Let $A$ be an associative $k$-algebra, $B$  a commutative $k$-algebra, and $\varepsilon:B\to A$  a morphism of $k$-algebras such that $\varepsilon(B)\subset {\mathcal Z}(A)$. Suppose that for each $n\in \mathbb{N}$ we have a $k$-linear map $c_n:A\otimes A\otimes B\to A$. For  $\alpha\in B$ we define $m_{\alpha,t}:A[[t]]\otimes A[[t]]\to A[[t]]$ determined by
 $$m_{\alpha,t}(a\otimes b)=\varepsilon(\alpha)ab+c_1(a\otimes b\otimes \alpha)t^1+c_2(a\otimes b\otimes \alpha)t^2+c_3(a\otimes b\otimes \alpha)t^3+...$$
 
 Notice that 
 \begin{eqnarray*}
 m_{\alpha+\beta,t}(a\otimes b)=m_{\alpha,t}(a\otimes b)+m_{\beta,t}(a\otimes b)
 \end{eqnarray*}
 and 
  \begin{eqnarray*}
 m_{q\alpha,t}(a\otimes b)=qm_{\alpha,t}(a\otimes b)
 \end{eqnarray*}
 for all $a$, $b\in A$, $\alpha$, $\beta\in B$ and $q\in k$.
 
We want to investigate under what conditions the family ${\mathcal M}=\{m_{\alpha,t}\}_{\alpha\in B}$ satisfies the following condition:
\begin{eqnarray}
m_{\alpha\beta,t}(id\otimes m_{\gamma,t})=m_{\beta\gamma,t}(m_{\alpha,t}\otimes id)
\label{genA2}
\end{eqnarray}
 
Take $a$, $b$, $c\in A$ and $\alpha$, $\beta$, $\gamma\in B$. If ${\mathcal M}$ satisfies (\ref{genA2}), then we get the  identity:
 \begin{eqnarray*} &&\varepsilon(\alpha)\varepsilon(\beta\gamma)(ab)c+c_1(\varepsilon(\alpha)ab\otimes c\otimes \beta\gamma)t+c_2(\varepsilon(\alpha)ab\otimes c\otimes \beta\gamma)t^2+...\\
&& +\varepsilon(\beta\gamma)c_1(a\otimes b\otimes \alpha)ct+c_1(c_1(a\otimes b\otimes \alpha)\otimes c\otimes \beta\gamma)t^2+...\\
&& +\varepsilon(\beta\gamma)c_2(a\otimes b\otimes \alpha)ct^2+...\\
 &&=\varepsilon(\alpha\beta)a\varepsilon(\gamma)(bc)+c_1(a\otimes \varepsilon(\gamma)bc\otimes \alpha\beta)t+c_2(a\otimes \varepsilon(\gamma)bc\otimes \alpha\beta)t^2+...\\&&+\varepsilon(\alpha\beta)ac_1(b\otimes c\otimes \gamma)t+c_1(a\otimes c_1(b\otimes c\otimes \gamma)\otimes \alpha\beta)t^2+...\\
 &&+\varepsilon(\alpha\beta)ac_2(b\otimes c\otimes \gamma)t^2+...
 \end{eqnarray*}

Since $\varepsilon$ is a morphism of $k$-algebras we have  that ${\mathcal M}$ is associative $mod \;t$ (i.e. it satisfies (\ref{genA2}) $mod \;t$). To get associativity $mod\; t^2$ the following 2-cocycle condition must be satisfied:
\begin{eqnarray}\label{cond2}
&\varepsilon(\alpha\beta)ac_1(b\otimes c\otimes \gamma)-c_1(\varepsilon(\alpha)ab\otimes c\otimes \beta\gamma)+c_1(a\otimes \varepsilon(\gamma)bc\otimes \alpha\beta)-&\\
&\varepsilon(\beta\gamma)c_1(a\otimes b\otimes \alpha)c=0.\nonumber&
\end{eqnarray}

Suppose that ${\mathcal M}$ is associative $mod \;t^2$ and we want to have associativity $mod \;t^3$, then $c_1$ and $c_2$ must satisfy  the condition:
\begin{eqnarray}\label{cond3}
&\varepsilon(\alpha\beta)ac_2(b\otimes c\otimes \gamma)-c_2(\varepsilon(\alpha)ab\otimes c\otimes \beta\gamma)+c_2(a\otimes \varepsilon(\gamma)bc\otimes \alpha\beta)-&\\
&\varepsilon(\beta\gamma)c_2(a\otimes b\otimes \alpha)c=c_1(c_1(a\otimes b\otimes \alpha)\otimes c\otimes \beta\gamma)-c_1(a\otimes c_1(b\otimes c\otimes \gamma)\otimes \alpha\beta)\nonumber
&
\end{eqnarray}
This suggest the construction of a cohomology theory that describes the obstruction to extending a family of products ${\mathcal M}$ that is associative $mod \;t^n$ to a family of products  that is associative $mod\; t^{n+1}$. 

\subsection{Secondary Cohomology of a Triple $(A,B,\varepsilon)$}

Let $A$ be an associative $k$-algebra, $B$  a commutative $k$-algebra, $\varepsilon:B\to A$ a morphism of $k$-algebras such that $\varepsilon(B)\subset {\mathcal Z}(A)$,  and $M$ an $A$-bimodule. Let  $$C^n((A,B,\varepsilon);M)=Hom_k(A^{\otimes n}\otimes B^{\otimes \frac{n(n-1)}{2}},M),$$
we need to define 
$$\delta^{\varepsilon}_n:C^n((A,B,\varepsilon);M)\to C^{n+1}((A,B,\varepsilon);M).$$
It is convenient to think about an element from   $A^{\otimes (n+1)}\otimes B^{\otimes \frac{n(n+1)}{2}}$ using the following matrix representation:
$$
\displaystyle\bigotimes
\left(
\begin{array}{cccccccc}
 a_{0}& b_{0,1} & b_{0,2} & ...&b_{0,n-2}&b_{0,n-1}&b_{0,n}\\
1 & a_{1}       & b_{1,2} &...&b_{1,n-2}&b_{1,n-1}&b_{1,n}\\
1& 1  & a_{2} &...&b_{2,n-2}      &b_{2,n-1}&b_{2,n}\\
. & .       &. &...&.&.&.\\
1& 1& 1  &...&1&a_{n-1}&b_{n-1,n}\\
1 & 1& 1  &...&1&1&a_{n}\\
\end{array}
\right),
$$
where $a_i\in A$ and $b_{i,j}\in B$. Take $f\in C^n((A,B,\varepsilon);M)$, we define: 
\begin{eqnarray*}&
\delta^{\varepsilon}_n(f)\left(\displaystyle\bigotimes
\left(
\begin{array}{cccccccc}
 a_{0}& b_{0,1} & b_{0,2} & ...&b_{0,n-2}&b_{0,n-1}&b_{0,n}\\
1 & a_{1}       & b_{1,2} &...&b_{1,n-2}&b_{1,n-1}&b_{1,n}\\
1& 1  & a_{2} &...&b_{2,n-2}      &b_{2,n-1}&b_{2,n}\\
. & .       &. &...&.&.&.\\
1& 1& 1  &...&1&a_{n-1}&b_{n-1,n}\\
1 & 1& 1  &...&1&1&a_{n}\\
\end{array}
\right)\right)=
\end{eqnarray*}
\begin{eqnarray*}
&
a_0\varepsilon(b_{0,1}b_{0,2}...b_{0,n})f\left(\displaystyle\bigotimes
\left(
\begin{array}{cccccc}
 a_{1}       & b_{1,2} &...&b_{1,n-1}&b_{1,n}\\
1  & a_{2} &...    &b_{2,n-1}&b_{2,n}\\
 .       &. &...&.&.\\
1& 1  &...&a_{n-1}&b_{n-1,n}\\
 1& 1  &...&1&a_{n}\\
\end{array}
\right)\right)-&\\
&\;&\;\\
&
f\left(\displaystyle\bigotimes
\left(
\begin{array}{ccccccc}
 \varepsilon(b_{0,1})a_{0}a_{1}       & b_{0,2}b_{1,2} &...&b_{0,n-1}b_{1,n-1}&b_{0,n}b_{1,n}\\
1  & a_{2} &...     &b_{2,n-1}&b_{2,n}\\
 .       &. &...&.&.\\
1& 1  &...&a_{n-1}&b_{n-1,n}\\
 1& 1  &...&1&a_{n}\\
\end{array}
\right)\right)+&\\
&\;&\\
& f\left(\displaystyle\bigotimes
\left(
\begin{array}{ccccccc}
 a_{0}       & b_{0,1}b_{0,2}&b_{0,3} &...&b_{0,n-1}&b_{0,n}\\
1  &\varepsilon(b_{1,2})a_{1}a_{2} &b_{1,3}b_{2,3}&...     &b_{1,n-1}b_{2,n-1}&b_{1,n}b_{2,n}\\
 .       &. &.&...&.&.\\
1& 1  &1&...&a_{n-1}&b_{n-1,n}\\
 1& 1  &1&...&1&a_{n}\\
\end{array}
\right)\right)-
\end{eqnarray*}
\begin{eqnarray*}
&.\;.\;.&\\
& \;&\\
& (-1)^if\left(\displaystyle\bigotimes
\left(
\begin{array}{cccccccc}
 a_{0}       & b_{0,1}&... &b_{0,i-1}b_{0,i}&...&b_{0,n-1}&b_{0,n}\\
1  &a_{1}&...&b_{1,i-1}b_{1,i}     &...&b_{1,n-1}&b_{1,n}\\
 .       &. &...&...&...&.&.\\
  1       &1 &...&\varepsilon(b_{i-1,i})a_{i-1}a_i&...&b_{i-1,n}b_{i,n}&b_{i-1,n}b_{i,n}\\
   .       &. &...&...&...&.&.\\
1& 1  &...&...&...&a_{n-1}&b_{n-1,n}\\
 1& 1  &...&...&...&1&a_{n}\\
\end{array}
\right)\right)+&\\
& \;&\\
&.\;.\;.&\\
& \;&
\end{eqnarray*}
\begin{eqnarray*}
&(-1)^{n-1}f\left(\displaystyle\bigotimes
\left(
\begin{array}{cccccc}
 a_{0}       & b_{0,1} &...&b_{0,n-2}b_{0,n-1}&b_{0,n}\\
1  &a_{1} &...     &b_{1,n-2}b_{1,n-1}&b_{1,n}\\
 .       &. &...&.&.\\
1& 1  &...&\varepsilon(b_{n-2,n-1})a_{n-2}a_{n-1}&b_{n-2,n}b_{n-1,n}\\
 1& 1  &...&1&a_{n}\\
\end{array}
\right)\right)+&\\
&\;&\\
&(-1)^{n}f\left(\displaystyle\bigotimes
\left(
\begin{array}{ccccccc}
 a_{0}       & b_{0,1} &...&b_{0,n-2}&b_{0,n-1}b_{0,n}\\
1  &a_{1} &...     &b_{1,n-2}&b_{1,n-1}b_{1,n}\\
 .       &. &...&.&.\\
1& 1  &...&a_{n-2}&b_{n-2,n}b_{n-1,n}\\
 1& 1  &...&1&\varepsilon(b_{n-1,n})a_{n-1}a_{n}\\
\end{array}
\right)\right)+&\\
&\;&\\
%\end{eqnarray*}
%\begin{eqnarray*}
&(-1)^{n+1}\varepsilon(b_{0,n}b_{1,n}...b_{n-1,n})f\left(\displaystyle\bigotimes
\left(
\begin{array}{cccccc}
 a_{0}       & b_{0,1} &...&b_{0,n-2}&b_{0,n-1}\\
1  &a_{1} &...     &b_{1,n-2}&b_{1,n-1}\\
 .       &. &...&.&.\\
1& 1  &...&a_{n-2}&b_{n-2,n-1}\\
 1& 1  &...&1&a_{n-1}\\
\end{array}
\right)\right)a_{n}&
\end{eqnarray*}

\begin{remark} If $n=2$ and $f\in C^2((A,B,\varepsilon);M)$ we have
\begin{eqnarray*}&
\delta^{\varepsilon}_2(f)\left(\displaystyle\bigotimes
\left(
\begin{array}{ccc}
 a_{0}& b_{0,1} & b_{0,2}\\
1 & a_{1}       & b_{1,2}\\ 
1&1&a_2\\
\end{array}
\right)\right)=&\\
&\varepsilon(b_{0,1}b_{0,2})a_0f\left(\displaystyle\bigotimes
\left(
\begin{array}{cc}
 a_{1}& b_{1,2}\\
1 & a_{2}      
\end{array}
\right)\right)-
f\left(\displaystyle\bigotimes
\left(
\begin{array}{cc}
\varepsilon(b_{0,1})a_0 a_{1}& b_{0,2}b_{1,2}\\
1 & a_{2}      
\end{array}
\right)\right)+&\\
&f\left(\displaystyle\bigotimes
\left(
\begin{array}{cc}
a_0 & b_{0,1}b_{0,2}\\
1 & \varepsilon(b_{1,2})a_1a_{2}      
\end{array}
\right)\right)-
\varepsilon(b_{0,2}b_{1,2})f\left(\displaystyle\bigotimes
\left(
\begin{array}{cc}
a_0 & b_{0,1}\\
1 & a_1      
\end{array}
\right)\right)a_2.&
\end{eqnarray*}
If $M=A$ then the equations (\ref{cond2}) and (\ref{cond3}) can be written as $\delta^{\varepsilon}_2(c_1)=0$ and $\delta^{\varepsilon}_2(c_2)=c_1 \overline{\circ}   c_1$, where $c_1 \overline{\circ} c_1$ is the right had side of equation \ref{cond3}. More generally for $f$,  $g:A\otimes A\otimes B\to A$  we define
\begin{eqnarray*}&
(f\overline{\circ} g)\left(\displaystyle\bigotimes
\left(
\begin{array}{ccc}
 a& \alpha & \beta \\
1 & b      & \gamma\\
1& 1  & c 
\end{array}
\right)\right)= 
\end{eqnarray*}
\begin{eqnarray*}
f\left(\bigotimes\left(
\begin{array}{cc}
g\left(\bigotimes\left(
\begin{array}{cc}
a & \alpha \\
1   & b\\
\end{array}
\right)\right)
& \beta\gamma \\
1   & c\\
\end{array}
\right)\right)-
f\left(\bigotimes\left(
\begin{array}{cc}
a& \alpha\beta \\
1   & g\left(\bigotimes\left(
\begin{array}{cc}
b& \gamma \\
1   & c\\
\end{array}
\right)\right)\\
\end{array}
\right)\right).
\end{eqnarray*}

\end{remark}
\begin{remark} Notice that $\delta_n^{\varepsilon}$ consists of $n+2$ terms. For $2\leq i\leq n+1$, the $i$-th term of that sum  is the result of evaluating $f$ on the tensor matrix that is obtained by multiplying  the entries in the $i$-th row (respectively $i$-th column) with the entries in the  $(i-1)$-st row (respectively $i-1$-st column). If we think that the $0$-th row and the $n+1$-st column consists of $1\otimes_A 1\otimes_A ...\otimes_A 1\in A\otimes_A A\otimes_A...\otimes_A A$, then a similar mnemonic rule can be used with the extra convention that $b\cdot 1=\varepsilon(b)\in {\mathcal Z}(A)$. 
\label{remark2}
\end{remark}

\begin{proposition} $(C^n((A,B,\varepsilon);M),\delta_n^{\varepsilon})$ is a complex (i.e. $\delta_{n+1}^{\varepsilon}\delta_n^{\varepsilon}=0$). We denote its homology by $H^n((A,B,\varepsilon);M)$ and we call it the secondary cohomology of the triple $(A,B,\varepsilon)$ with coefficients in $M$.
\end{proposition}
\begin{proof}
It follows from Remark \ref{remark2} from above; see also  the discussion from the next section.
\end{proof}

\begin{example} When $B=k$ and $\varepsilon:k\to A$ we have that $H^n((A,k,\varepsilon);M)$ is the usual Hochschild Cohomology. 
\end{example}
\begin{example} Let $G$ be a group, $H$ an abelian group, and $K$ a $G$-module. Take $A=k[G]$, $B=k[H]$, $\varepsilon:B\to A$ determined by $\varepsilon(\alpha)=1$ for all $\alpha\in H$,  and $M=k[K]$. One can see that $H^n((A,B,\varepsilon);M)$ is the "linearization" of the secondary cohomology introduced in \cite{staic}. 
\end{example}

\begin{theorem} Let $A$ be a $k$-algebra, $B$ a commutative $k$-algebra, $\varepsilon:B\to A$ a morphism of $k$-algebras such that $\varepsilon(B)\subset {\mathcal Z}(A)$.  Suppose that for all $n\geq 1$ we $c_n:A\otimes A\otimes B\to A$ $k$-linear maps. We consider  a family of products
${\mathcal M}=\{m_{\alpha,t}\}_{\alpha\in B}$ where 
 $m_{\alpha,t}(a\otimes b)=\varepsilon(\alpha)ab+c_1(a\otimes b\otimes \alpha)t^1+c_2(a\otimes b\otimes \alpha)t^2+c_3(a\otimes b\otimes \alpha)t^3+...$

i) The  family  ${\mathcal M}$ is associative $mod\; t^2$ if and only if $c_1\in Z^2((A,B,\varepsilon);A)$. Moreover $c_1\in H^2((A,B,\varepsilon);A)$ is determined by the isomorphism class of ${\mathcal M}$.

ii) Suppose that ${\mathcal M}$ is associative $mod\; t^{n+1}$, then ${\mathcal M}$ can be extended to a family of products that is associative $mod\; t^{n+2}$ if and only if $c_1\overline{\circ}  c_n+c_2\overline{\circ} c_{n-1}+...+c_n\overline{\circ}  c_1=0\in H^3((A,B,\varepsilon);A)$.
\end{theorem}
\begin{proof}  The proof follows from the above discussion and standard arguments for any deformation theory. 
We will only give details for  the fact that  the class of $c_1\in H^2((A,B,\varepsilon);A)$ is determined by the isomorphism class of ${\mathcal M}$. 

Consider two families of products $\{m_{\alpha,t}^c\}_{\alpha\in B}$ and $\{m_{\alpha,t}^d\}_{\alpha\in B}$ on $A[[t]]$, $$m_{\alpha,t}^c(a\otimes b)=\varepsilon(\alpha)ab+c_1(a\otimes b\otimes \alpha)t^1+...$$
 and 
 $$m_{\alpha,t}^d(a\otimes b)=\varepsilon(\alpha)ab+d_1(a\otimes b\otimes \alpha)t^1+...$$ S
 uppose that we have $f:A[[t]]\otimes A[[t]]$, 
 $$f(a)=a+f_1(a)t+f_2(a)t^2+...$$ 
 an isomorphism, such that $m_{\alpha,t}^d(f(a)\otimes f(b))=f(m_{\alpha,t}^c(a\otimes b))$. If we want this identity to be true $mod\; t^2$ we get the identity:
\begin{eqnarray*}
&\varepsilon(\alpha)ab+d_1(a\otimes b\otimes \alpha)t +\varepsilon(\alpha)af_1(b)t+\varepsilon(\alpha)f_1(a)bt+O(t^2)=\\
&\varepsilon(\alpha)ab+f(\varepsilon(\alpha)ab)t+c_1(a\otimes b\otimes \alpha)t+O(t^2).
\end{eqnarray*}
This is equivalent to 
\begin{eqnarray*}
\varepsilon(\alpha)af_1(b)-f(\varepsilon(\alpha)ab)+\varepsilon(\alpha)f_1(a)b=c_1(a\otimes b\otimes \alpha)-d_1(a\otimes b\otimes \alpha)
\end{eqnarray*}
or 
$$c_1-d_1=\delta_1^{\varepsilon}(f_1)$$ 
And so,  $c_1$ and $d_1$ have the same class in $H^2((A,B,\varepsilon);A)$.
\end{proof}

\begin{remark} If  $(A[[t]],m_{1,t})$ admits a unit $1\in A[[t]]$ then it follows from  Proposition \ref{propfam} that  $(A[[t]],m_{1,t})$ is a $B$-algebra with the map $\overline{\varepsilon}:B\to A[[t]]$, $\overline{\varepsilon}(\alpha)=m_{\alpha}(1\otimes 1)$.
\end{remark}

\begin{remark} There is a trivial way of getting an $B$-algebra structure on $A[[t]]$. More precisely, if $c_i:A\otimes A\to A$ have the property $c_i(a\otimes 1)=c_i(1\otimes a)=0$ for all $i\geq 1$ and $c_i$ are $B$-linear, then there exists a natural inclusion  $B\to A[[t]]$  which gives a $B$-algebra structure on $A[[t]]$. This coresponds to $\tilde{c}_i:A\otimes A\otimes B\to A$, $\tilde{c}_i(a\otimes b\otimes \alpha)=\varepsilon(\alpha)c_i(a\otimes b)$. 
\end{remark}

\begin{remark}  There is a natural map $r:H^n((A,B,\varepsilon);M)\to H^n(A,M)$. When $n=2$ and $n=3$ this map corresponds to forgetting the $B$-algebra structure on $A[[t]]$.
\end{remark}

\section{The Second Cyclic Module of a Commutative Algebra}

In this section $B$ is a commutative $k$-algebra and $\varepsilon :B\to k$ is a morphism of $k$-algebras.

When we want to define a cyclic map $\tau_n$ on the $k$-module $B^{\otimes \frac{n(n-1)}{2}}$, the main problem is that  we do not have  enough components to get a cyclic action. That is because there is no natural action of the cyclic group $C_{n+1}$ on the index set $\{(u,v)\vert \;0\leq u<v\leq n-1\}$. To fix this problem  we have to  find a set that contains the above index set, and admits a natural action of the cyclic group $C_{n+1}$.  Let  $\tau_n:\{0,1,...,n\}\to \{0,1,...,n\}$ defined by $\tau_n(i)=i+1$ (with the convention that $\tau(n)=n+1=0$). 

We take   $I_n=\{(u,v)\vert \; 0\leq u,v\leq n,\; u\neq v,\; u\neq v-1\}$. One can see that we have an action of the cyclic group $C_{n+1}$ on $I_n$ given by
$$(u,v)\to (\tau_n(u),\tau_n(v)).$$
We define  
$$K_n=B^{\otimes I_n}.$$

Just like in the previous section, the best way to see an element of $K_n$ is to think about it as a $(n+1)\times (n+1)$ tensor matrix for which certain entries are 1, more exactly we will take $b_{i,i}=1$ and $b_{\tau_n(i),i}=1$. 

$$
\displaystyle\bigotimes_{(u,v)\in I_n}
\left(
\begin{array}{cccccccc}
1 & b_{0,1} & b_{0,2} & b_{0,3}&...&b_{0,n-2}&b_{0,n-1}&1\\
1 & 1       & b_{1,2} &b_{1,3} &...&b_{1,n-2}&b_{1,n-1}&b_{1,n}\\
b_{2,0}& 1  & 1 &b_{2,3} &...&b_{2,n-1}      &b_{2,n-2}&b_{2,n}\\
. & .       & . &. &...&.&.&.\\
b_{n-1,0} & b_{n-1,1}& b_{n-1,2} &b_{n-1,3} &...&1&1&b_{n-1,n}\\
b_{n,0} & b_{n,1}& b_{n,2} &b_{n,3} &...&b_{n,n-2}&1&1\\
\end{array}
\right)
$$
As one can easily see there are $n^2-1$ nontrivial entries in this tensor matrix. Notice that we put $1$ on the trivial entries and not $0$ since the matrix represents a tensor. However, we point out that this $1$ is not the identity of $B$, but rather a generator of the $B$-module $k$, where $k$ is an $B$-module with  $b\cdot 1=\varepsilon(b)1$.   

First define the action of the cyclic group $C_{n+1}$ $\tau_n:K_n\to K_n$
$$\tau_n(\otimes_{(u,v)\in I_n}b_{u,v})=\otimes_{(u,v)\in I_n}b_{\tau_n(u),\tau_n(v)}$$
It is easy to see that in the tensor-matrix notation this corresponds to shifting simultaneously the rows and columns with one unit (the last row and last column become the first row respectively first column).
For obvious reasons one has $\tau_n^{n+1}=id$.

Next define $\partial_n:K_n\to K_{n-1}$ 
$$\partial_n(\displaystyle\bigotimes\limits_{(u,v)\in I_n}b_{u,v})=\varepsilon(b_{n-1,n}b_{n,n-2}b_{0,n-1})\displaystyle\bigotimes\limits_{(u,v)\in I_{n-1}}d_{u,v}$$
where 
$$d_{u,v}= \left\{
\begin{array}{c l}
b_{u,v}\; &{\rm if}\; u\neq n-1\; {\rm and} \;v \neq n-1\\
b_{u,n-1}b_{u,n}\; &{\rm if}\; v= n-1\\
b_{n-1,v}b_{n,v}\; &{\rm if}\; u= n-1 
\end{array}
\right.
$$
or in the tensor matrix form

\begin{eqnarray*}&
\partial_n\left(\displaystyle\bigotimes
\left(
\begin{array}{ccccccc}
1 & b_{0,1} & b_{0,2} &...&b_{0,n-2}&b_{0,n-1}&1\\
1 & 1       & b_{1,2} &...&b_{1,n-2}&b_{1,n-1}&b_{1,n}\\
b_{2,0}& 1  & 1 &...&b_{2,n-1}      &b_{2,n-2}&b_{2,n}\\
. & .       & . &...&.&.&.\\
b_{n-1,0} & b_{n-1,1}& b_{n-1,2} &...&1&1&b_{n-1,n}\\
b_{n,0} & b_{n,1}& b_{n,2} &...&b_{n,n-2}&1&1\\
\end{array}
\right)\right)=
\end{eqnarray*}
\begin{eqnarray*}&
\displaystyle\bigotimes
\left(
\begin{array}{ccccccc}
1 & b_{0,1} & b_{0,2} & ...&b_{0,n-2}&\varepsilon(b_{0,n-1})\\
1 & 1       & b_{1,2} &...&b_{1,n-2}&b_{1,n-1}b_{1,n}\\
b_{2,0}& 1  & 1 &...&b_{2,n-1}      &b_{2,n-2}b_{2,n}\\
. & .       & . &...&.&.\\
b_{n-2,0} & b_{n-2,1}& b_{n-2,2}&...&1&b_{n-2,n-1}b_{n-2,n}\\
b_{n-1,0}b_{n,0} & b_{n-1,1}b_{n,1}& b_{n-1,2}b_{n,2} &...&\varepsilon(b_{n,n-2})&\varepsilon(b_{n-1,n})\\
\end{array}
\right)
\end{eqnarray*}

In the tensor matrix notation $\partial_n$ corresponds to collapsing simultaneously the $n$-th row and $n$-th column in the $(n-1)$-st row and $(n-1)$-st column. Here we use the convention that the product between $b\in B$ and $1\in k$ is $\varepsilon(b)1\in k$.  

Also we define $s_n:K_n\to K_{n+1}$
$$s_n(\displaystyle\bigotimes\limits_{(u,v)\in I_n}b_{u,v})=\displaystyle\bigotimes\limits_{(u,v)\in I_{n+1}}c_{u,v}$$
where 
$$c_{u,v}= \left\{
\begin{array}{c l}
b_{u,v}\; &{\rm if} \;u\neq n+1\; {\rm and} \;v \neq n+1\\
1\; &{\rm if}\;  u=n+1\; {\rm or} \;v = n+1\\
\end{array}
\right.
$$

or in the tensor matrix form
\begin{eqnarray*}&
s_n\left(\displaystyle\bigotimes
\left(
\begin{array}{ccccccc}
1 & b_{0,1} & b_{0,2} & ...&b_{0,n-2}&b_{0,n-1}&1\\
1 & 1       & b_{1,2} &...&b_{1,n-2}&b_{1,n-1}&b_{1,n}\\
b_{2,0}& 1  & 1 &...&b_{2,n-1}      &b_{2,n-2}&b_{2,n}\\
. & .       & . &...&.&.&.\\
b_{n-1,0} & b_{n-1,1}& b_{n-1,2} &...&1&1&b_{n-1,n}\\
b_{n,0} & b_{n,1}& b_{n,2} &...&b_{n,n-2}&1&1\\
\end{array}
\right)\right)=
\end{eqnarray*}
\begin{eqnarray*}&
\displaystyle\bigotimes
\left(
\begin{array}{cccccccc}
1 & b_{0,1} & b_{0,2} & ...&b_{0,n-2}&b_{0,n-1}&1&1\\
1 & 1       & b_{1,2} &...&b_{1,n-2}&b_{1,n-1}&b_{1,n}&1\\
b_{2,0}& 1  & 1 &...&b_{2,n-1}      &b_{2,n-2}&b_{2,n}&1\\
. & .       & . &...&.&.&.&.\\
b_{n-1,0} & b_{n-1,1}& b_{n-1,2} &...&1&1&b_{n-1,n}&1\\
b_{n,0} & b_{n,1}& b_{n,2} &...&b_{n,n-2}&1&1&1\\
1 & 1& 1 &...&1&1&1&1\\
\end{array}
\right)
\end{eqnarray*}
 
In the tensor matrix notation $s_n$ corresponds to adding one more row and one more column of $1$. More precisely the elements from the position $(n+1,n)$, $(n+1,n+1)$ and $(0,n+1)$ are $1\in k$ while the rest are $1\in B$.  
Next we define $\partial_i:K_n\to K_{n-1}$ and $s_i:K_n\to K_{n+1}$
\begin{eqnarray}
\partial_i=\tau_{n-1}^{i-n}\partial_n\tau_n^{n-i}\label{pari}
\end{eqnarray}
\begin{eqnarray}
s_i=\tau_{n+1}^{i-n}s_n\tau_n^{n-i}\label{si}
\end{eqnarray}
One can see that in the tensor matrix notation $\partial_i$ correspond to collapsing simultaneously the $i$-th row (respectively $i$-th column) into the $i-1$-st row (respectively $i-1$-st column). Also $s_i$ consist of inserting a row and a column of $1$ after $i$-th row, respectively $i$-th column. Using these interpretations, it is easy to show that
\begin{eqnarray*}
&\partial_i\partial_j=\partial_{j-1}\partial_i \; {\rm if}\; i<j\\
&s_is_j=s_{j+1}s_i \; {\rm if}\; i\leq j,
\end{eqnarray*}
and
$$\partial_i s_j= \left\{
\begin{array}{c l}
s_{j-1}\partial_i\; &{\rm if}\; i<j\\
id_{K_n}&{\rm if}\; i=j\;{\rm or }\; i=j+1\\
s_j\partial_{i-1} &{\rm if}\;  i>j+1.
\end{array}
\right.$$
Finally, using \ref{pari} and \ref{si} we get 
\begin{eqnarray*}
\partial_i\tau_n=\tau_{n-1}\partial_{i-1}\\
s_i\tau_n=\tau_{n+1}\partial_{i-1}.
\end{eqnarray*}

\begin{theorem} $\,_2K(B)=(B^{\otimes (n^2-1)},\partial_i,s_i\tau_n)$ is a cyclic $k$-module.
\end{theorem}
\begin{proof}
It follows from the above discussion.
\end{proof}

\begin{remark} One can use a similar idea  to construct a cyclic module associated to a triple $(A,B,\varepsilon)$ as in previous section. In that situation we need to define $\,_2K(A,B,\varepsilon)_n=A^{\otimes (n+1)}\otimes B^{\otimes (n^2-1)}$.
\end{remark}

%%%%%%%%%%%%%%%
%\section*{Acknowledgment}

%%%%%%%%%%%%%%%%%%%%%%%%%%%%%
%%%%%%%%%%%%%%%%%%%%%%%%%%%%%%%%%%%%%%%
%%%%%%%%

%%%%%%%%%%

\bibliographystyle{amsalpha}

\end{document}